\documentclass[preprint]{elsarticle}
\journal{J. Fixed Point Theory Appl.}

\usepackage{enumerate}  
\usepackage{amsmath}    
\usepackage{amsthm}     
\usepackage{mathrsfs}   
\usepackage{amsfonts}
\usepackage{amssymb}
\usepackage{mathtools} 	
\usepackage{verbatim}	

\usepackage[T1]{fontenc}

\theoremstyle{plain}                
\newtheorem{thm}{Theorem}[section]  
\newtheorem{lem}[thm]{Lemma}        
\newtheorem{prop}[thm]{Proposition}
\newtheorem{cor}[thm]{Corollary}

\theoremstyle{definition}
\newtheorem{defn}[thm]{Definition}

\theoremstyle{remark}
\newtheorem*{rem}{Remark}

\DeclareMathOperator{\diam}{diam}   
\DeclareMathOperator{\cl}{cl}
\DeclareMathOperator*{\conv}{conv}  

\providecommand{\norm}[1]{\lVert#1\rVert}

\begin{document}

\begin{frontmatter}

\title{Exact controllability of non-Lipschitz semilinear systems }
\author[uol]{Rados{\l}aw Zawiski\corref{correspondingauthor}}
\ead{R.Zawiski@leeds.ac.uk}
\cortext[correspondingauthor]{Corresponding author}
\address[uol]{School of Mathematics, University of Leeds, LS2 9JT Leeds, UK}
\begin{abstract}
We present sufficient conditions for exact controllability of a semilinear infinite dimensional dynamical system. The system mild solution is formed by a noncompact semigroup and a nonlinear disturbance that does not need to be Lipschitz continuous. Our main result is based on a fixed point type application of the Schmidt existence theorem and illustrated by a nonlinear transport partial differential equation.
\end{abstract}


\end{frontmatter}

\section{Introduction}

Controllability of nonlinear systems is a mature subject of research - see \cite{BalDau02,ObuZec09,CarQui84} and references therein. In recent years various applications of fixed point theorems are particularly popular among researchers tackling this problem. These range from classical Banach or  Schauder Fixed Point Theorems (FPT) to more specific, such as Nussbaum \cite{Nus69} FPT in \cite{CarQui84}, Schaefer \cite{Scha55} FPT in \cite{Choi_Sakthivel_2004} or M{\"o}nch \cite{Mon80} FPT in \cite{RavMach13}. A short survey on fixed point approaches is given in \cite{Zaw14}.

In most cases the nonlinearities present in (otherwise linear) systems are regarded as disturbances, in some way influencing the normal operation of a system. The choice of the fixed-point-type approach depends on the nature of nonlinearity and the structure of the system itself. The examples of such approach we particularly focus on, are given in \cite{CarQui84} and \cite{RavMach13}. In the former case the authors make use of the fact that the system under consideration can be represented as a sum of Lipschitz-type and compact operators, what allows them to apply the Nussbaum \cite{Nus69} fixed point theorem. In the latter case the authors examine the controllability conditions of a semilinear impulsive mixed Volterra-Fredholm functional integro-differential evolution differential system with finite delay and nonlocal conditions by means of measures of noncompactness and M{\"o}nch fixed point theorem \cite{Mon80}.

Interesting, however, is that although almost 30 years separate these articles they both contain assumption of the Lipschitz type of nonlinearity. The author of this article is actually not aware of any example of fixed point theorem application to the problem of exact controllability where the nonlinearity is not Lipschitz in some way. The reason for that is that with Lipschitz condition comes either computational 'smoothness', which goes back to existence results for initial value problem such as e.g. Picard-Lindel{\"o}f, or equicontinuity needed by the Ambrosetti Theorem to express the measure of noncompactness. 

For this reason we intentionally drop the Lipschitz condition. In this way this article combines and expands above results by means of the Schmidt existence theorem, originally developed for the Cauchy problem in Banach spaces \cite{Schmidt92}. This requires a reformulation of the theorem into the fixed point form. The set of assumptions is discussed and the results follow. In particular, we do not impose any compactness condition. The article is finished with an illustrative example.

\section{Preliminaries}

This section gives the basic definitions and background material. It also defines the notation. If for lemmas or theorems given without reference to a particular source the proof is short and simple, they are immediately followed by the $\square$ sign.
\subsection{On measures of noncompactness, one-sided Lipschitz condition and Schmidt Theorem}
\begin{lem}
Let $\Xi$ be a normed space and $x,y\in\Xi$. Then the real function $p:\mathbb{R}\rightarrow[0;+\infty)$, $p(t):=\|x+ty\|$ is convex.$\qed$
\end{lem}

\begin{cor}
If for given $t\in\mathbb{R}$ the left hand side derivative $p_{-}'$ of $p$ given by above Lemma exists at point $t$ and the right hand side derivative $p_{+}'$ also exists at point $t$,
then the inequality $p_{-}'(t)\leq p_{+}'(t)$ holds.$\qed$
\end{cor}

\begin{defn}[One-sided Lipschitz condition]\label{defn one-sided Lipschitz cond}
Let $\Xi$ be a Banach space and $x,y\in \Xi$. We define the symbol
$$[x,y]_{\pm}:=\lim_{h\rightarrow 0^{\pm}}\frac{\|x+hy\|-\|x\|}{h}$$
and say that a function $f:\Xi\rightarrow\Xi$ fulfills \textit{one-sided Lipschitz condition} (left "-" or right "+", respectively) if there exists a non-negative constant $M$ such that
$$[x-y,g(x)-g(y)]_{\pm}\leq M\|x-y\|$$
for any $x,y\in\Xi$.
\end{defn}

\begin{lem}\label{lem bracket properties}
If the limit in Definition~\ref{defn one-sided Lipschitz cond} exists, then $[x,y]_{-}=p_{-}'(0)$, $[x,y]_{+}=p_{+}'(0)$ and
\begin{itemize}
\item [a)]$[x,y]_{-}\leq [x,y]_{+}$
\item [b)]$|[x,y]_{\pm}|\leq \|y\|$
\item [c)]$[0,y]_{\pm}=\pm\|y\|$
\item [d)]$[x,y+z]_{\pm}\leq [x,y]_{\pm}+\|z\|$.
\end{itemize}
\end{lem}
\begin{proof}
Proofs of points a) to c) follow directly from the bracket definition above. We will show only the last one.
\begin{itemize}
\item [1.] Consider the case $h<0$:
$$\|x+h(y+z)\|\geq\|x+hy\|-\|hz\|=\|x+hy\|+h\|z\|$$
and the case $h>0$:
$$\|x+h(y+z)\|\leq\|x+hy\|+\|hz\|=\|x+hy\|+h\|z\|$$
\item [2.] In both estimations in 1. by subtracting $\|x\|$ from both sides and dividing, respectively, by $h<0$ or $h>0$ one obtains
$$\frac{\|x+h(y+z)\|-\|x\|}{h}\leq\frac{\|x+hy\|-\|x\|}{h}+\|z\|$$
\item [3.] Going to the limit in 2. the result follows.
\end{itemize}
\end{proof}

\begin{lem}\label{thm one-sided Lipshitz properties}
Let $\Xi$ be a Banach space, $f:\mathbb{R}\times\Xi\rightarrow\Xi$ and $M$ be a nonnegative constant. Introducing the notation
\begin{itemize}
\item [($l_{-}$)] $[x-y,f(t,x)-f(t,y)]_{-}\leq M\|x-y\| \quad \forall x,y\in\Xi$
\item [($l_{+}$)] $[x-y,f(t,x)-f(t,y)]_{+}\leq M\|x-y\| \quad \forall x,y\in\Xi$
\item [($l$)] $\|f(t,x)-f(t,y)\|\leq M\|x-y\| \quad \forall x,y\in\Xi$
\end{itemize}
the chain of implications $(l)\Rightarrow(l_{+})\Rightarrow(l_{-})$ is true.
\end{lem}
\begin{proof}
Based on the previous lemma the following estimation holds
$$[x-y,f(t,x)-f(t,y)]_{-}\leq[x-y,f(t,x)-f(t,y)]_{+}\leq\|f(t,x)-f(t,y)\|\leq M\|x-y\|.$$
\end{proof}

\begin{lem}\label{lem one-sided Lipschitz in Hilbert space}
 Let $\Xi$ be a real inner product space. Then 
\begin{itemize}
 \item [a)] $[x,y]_{\pm}=\frac{1}{\norm{x}}\langle x,y\rangle\qquad \forall x,y\in\Xi,\ x\neq0.$
 \item [b)] Suppose $D\subset\mathbb{R}\times\Xi$ and $f:D\rightarrow\Xi$. Then both conditions
  \begin{itemize}
   \item [($l_\pm$)] $[x-y,f(t,x)-f(t,y)]_{\pm}\leq M\|x-y\|\qquad\forall (t,x),(t,y)\in D$
  \end{itemize}
  are equicvalent to 
  \begin{equation*}\label{eqn one-sided Lipschitz cond in Hilbert space}
    \langle x-y,f(t,x)-f(t,y)\rangle\leq M\norm{x-y}^2\qquad\forall (t,x),(t,y)\in D.
  \end{equation*}
\end{itemize}
\end{lem}
\begin{proof}
Proof follows from a straightforward calculation using the Definition \ref{defn one-sided Lipschitz cond}.
\end{proof}

Before proceeding further we recall 
\begin{defn}[Diameter of a set]\label{defn diameter of a set}
 Let $ \Xi $ be a metric space with a metric $ \rho $.  The \textit{diameter of a set} $A\subseteq\Xi$ is defined as
\[
\diam A:=\sup_{x,y\in A} \rho(x,y)\leq\infty 
\]
For the case of an empty set we take $\diam\emptyset=0$.
\end{defn}

We can now introduce one of the most commonly used measures of noncompactness (MNC), namely
\begin{defn}[Kuratowski measure of noncompactness \cite{Kuratowski_1930}]\label{defn Kuratowski mnc}
For a bounded subset $ A $ of a metric space $ \Xi $ we call
$$ \alpha(A):=\inf\{\delta\geq0:A\subseteq\bigcup_{i=1}^n A_{i};\ \diam A_{i}\leq\delta,\ i=1,\dots,n;\ n\in N \} $$
the \textit{Kuratowski MNC}.
\end{defn}

%

The Kuratowski MNC has properties given by the following
\begin{thm}[Properties of the Kuratowski MNC \cite{Darbo_1955,Appell_2005}]\label{thm mnc properties}
For bounded $A,B\subseteq\Xi$ and $\alpha$ MNC we have
\begin{itemize}
\item [a)]$\alpha(A)\leq \diam A$
\item [b)]if $A\subseteq B$ then $\alpha(A)\leq\alpha(B)$
\item [c)]$\alpha(A\cup B)=\max\{\alpha(A),\alpha(B)\}$
\item [d)]$\alpha(\cl A)=\alpha(A) $ where $\cl$ stands for closure
\item [e)]if $\Xi$ is a normed space and $\dim\Xi=\infty$ then $0\leq\alpha(B(0,1))\leq2$
\end{itemize}
\end{thm}

Additionally, if $\Xi$ is a Banach space, than the following Theorem is true \cite{Banas_Goebel}.
\begin{thm}\label{thm mnc properties B-space}
For bounded subsets $A,B$ of a Banach space $\Xi$, a constant $m$ and a MNC $\alpha$ there is
\begin{itemize}
\item [f)]$\alpha(A)=0$ if and only if $A$ is relatively compact
\item [g)]$\alpha(A+B)\leq\alpha(A)+\alpha(B)$
\item [h)]$\alpha(mA)=|m|\alpha(A)$
\item [i)]$\alpha(\conv A)=\alpha(A)$ where $\conv$ stands for convex hull
\end{itemize}
\end{thm}

In the sequel we will need \cite{Breckner_1984} the following  Definitions and the Theorem on equicontinuity of a set of functions.

\begin{defn}\label{defn ordered linear space}
 An \textit{ordered linear space} $Y$ is a space $Y$ on which there is defined a binary relation $\leq$ such that for all $x,y,z\in Y$ the following conditions are satisfied
\begin{itemize}
  \item[a)] $x\leq x$
  \item[b)] $x\leq y$ and $y\leq z$ imply $x\leq z$
  \item[c)] $x\leq y$ implies $x+z\leq y+z$
  \item[d)] $x\leq y$ implies $ax\leq ay$ for all real numbers $a>0$.
\end{itemize}
\end{defn}

\begin{defn}\label{defn wedge}
 A \textit{wedge} $C$ is a nonempty subset of a liner space $Y$ satisfying
\[
 aC+bC\subseteq C\qquad \forall a,b\in[0,\infty).
\]
A \textit{positive wedge} of an ordered linear space $Y$ is the set $Y_+$ of all elemets $x\in Y$ such that $0\leq x$, where $0$ denotes the zero element of $Y$.

We see that $Y_+$ is a wedge. Conversely, if $C$ is a wedge in a real linear space $Y$, then the binary relation $\leq$ given by
\begin{equation}\label{eqn oredering induced by C}
x\leq y\text{ if } y-x \in C
\end{equation}
satisfies all conditions in Definition \ref{defn ordered linear space} for all $x,y,z\in Y$, and in consequence makes $Y$ into an ordered linear space whose positive wedge is exactly $C$. The relation $\leq$ defined by \eqref{eqn oredering induced by C} is called \textit{the ordering induced by} $C$.

Let $Y$ be a topological linear space. Then $C$ is said to be a \textit{normal wedge} if for each neighbourhood $W$ of $0$ in $Y$ there exists a neighbourhood $V$ of $0$ in $Y$ such that 
\[
 (V-C)\cap(V+C)\subseteq W.
\]
\end{defn}

\begin{defn}\label{defn convex function}
 Let $M$ be a convex subset of a linear space $X$ and $Y$ be an ordered linear space. Then $f:M\rightarrow Y$ is called a \textit{convex function} when for all $a\in[0,1]$ and $x,y\in M$ the inequality
\[
 f(ax+(1-a)y)\leq af(x)+(1-a)f(y)
\]
holds. When the order on $Y$ is induced by a wedge $C$, the above can be written as 
\[
 af(x)+(1-a)f(y)\in f(ax+(1-a)y)+C.
\]
\end{defn}

The following Theorem, taken from \cite{Kosmol_et_al_1979}, is a generalization of the well known Banach-Steinhaus Theorem \cite[Theorem 2.5]{Rudin_fun}.

\begin{thm}\label{thm equicontinuity theorem by Kosmol}
Let $M$ be an open convex subset of a topological vector space $X$ of the second category, let $Y$ be a topological vector space ordered by a normal wedge $C$ and let $\mathcal{F}$ be a pointwise bounded family of continuous convex operators $f:M\rightarrow Y$. Then $\mathcal{F}$ is equicontinuous.
\end{thm}

A generalization of Theorem \ref{thm equicontinuity theorem by Kosmol} to the class of \textit{s-convex functions}, containing a necessary and sufficient condition of equicontinuity, can be found in \cite{Breckner_1984}, and is further developed in \cite{Breckner_Trif_1999}.

In expressing MNC in function space a key role is played by the following
\begin{thm}[Ambrosetti Theorem \cite{Ambro67}]\label{thm Ambrosetti's MNC}
Suppose that $J$ is a compact interval, $\mathcal{F}\subset C(J,E)$, $E$ is a Banach space, $\mathcal{F}$ is bounded and equicontinuous. Then
$$\alpha(\mathcal{F})=\sup_{t\in J}\alpha\big(\mathcal{F}(t)\big)=\alpha\big(\mathcal{F}(J)\big).$$
\end{thm}
We also make use of the following
\begin{defn}\label{defn compact operator}
Let $\Xi_{1}$ and $\Xi_{2}$ be metric spaces, $\Phi:\Xi_{1}\rightarrow\Xi_{2}$ be continuous and mapping bounded sets $A\subseteq\Xi_{1}$
onto bounded sets $\Phi(A)\subseteq\Xi_{2}$. If there exists a constant $M\geq0$ such that for every bounded $F\subset\Xi_{1}$ the inequality
$$\alpha(\Phi(F))\leq M\alpha(F) $$
holds, then $\Phi$ is called an $\alpha$\textit{-condensing operator} with constant $M$.
\end{defn}

The main tool we will use to prove our results is given by
\begin{thm}[Schmidt Theorem for the Initial Value Problem \cite{Schmidt92,Vol95}]\label{thm Schmidt ivp_existence}
Let $X$ be a Banach space and $T,\ M_{g},\ M_{k}$ be reals. Suppose $g,k:[0,T]\times X\rightarrow X$ are continuous, bounded and
\begin{itemize}
\item [a)]$[x_1-x_2,g(t,x_1)-g(t,x_2)]\_\leq M_{g}\|x_1-x_2\|\qquad\forall t\in[0,T],\ \forall x_1,x_2\in X$
\item [b)]$\alpha(k([0,T],D))\leq M_{k}\alpha(D)\qquad\forall D\subseteq X$, $D$ bounded.
\end{itemize}
Then the initial value problem (IVP)
\begin{equation}\label{eqn Schmidt_IVP}
\left\{\begin{array}{l}
        \frac{d}{dt}x(t)=g(t,x(t))+k(t,x(t)) \\
        x(0)=0
       \end{array}
\right.
\end{equation}

has a solution $x:[0,T]\rightarrow E$.
\end{thm}
A function $g$ with properties as above will be called \textit{dissipative with constant $M_g$}, or \textit{dissipative with $M_g$} for short, and a function $k$ with properties as above will be called \textit{condensing with constant $M_k$} or \textit{condensing with $M_k$}.

We will use an integral form of~\eqref{eqn Schmidt_IVP}, as it better suits our needs, that is
\begin{equation}\label{eqn Schmidt integral form}
x(t)=\int_{0}^{t}g\big(s,x(s)\big)ds+\int_{0}^{t}k\big(s,x(s)\big)ds,\quad t\in [0,T],
\end{equation}
as every solution of~\eqref{eqn Schmidt_IVP} is also a solution of~\eqref{eqn Schmidt integral form}. 
\subsection{On dynamical systems}
From this point onward we drop the general Banach space setting. Although some of the definitions make sense and the results are true, the Hilbert space setting allows us to obtain more concrete results. Hence, throughout the rest of this paper, $X$ and $U$ are Hilbert spaces which are identified with their duals.  For the whole remaining part $J:=[0,T]$ is a compact interval.

We will also use the Sobolev space of vector valued functions 
\[
H^1(J,X)=W^{1,2}(J,X):=\{f\in L^2(J,X):\frac{d}{dt}f(t)\in L^2(J,X)\}.
\]
Let $A:D(A)\rightarrow X$ be a densly defined, linear, closed and unbounded operator on which the Cauchy problem of interest is based. Before introducing the Cauchy problem formally we describe the setting in which it will be considered.

 Basic properties of a generator of a strongly continuous semigroup are gathered in the proposition below \cite[Theorem 1.2.4]{Pazy}:

\begin{prop}\label{prop semigroup generator properties}
Let $(Q(t))_{t\geq0}$ be a strongly continuous semigroup and let $(A,D(A))$ be its generator. Then
\begin{itemize}
  \item[a)] there exist constants $\omega\geq0$ and $M\geq1$ such that for every $t\geq0$ there is 
    \[
    \norm{Q(t)}\leq Me^{\omega t},
    \]
  \item[b)] for every $x\in X$ the function $t\mapsto Q(t)x$ is continuous from $[0,\infty)$ into $X$,
  \item[c)] for every $x\in X$
    \[
     \lim_{h\rightarrow0}\frac{1}{h}\int_{t}^{t+h}Q(s)xds=Q(t)x,
    \]
  \item[d)] for every $x\in X$ there is $\int_{0}^{t}Q(s)xds\in D(A)$ and 
\[
 A\int_{0}^{t}Q(s)xds=Q(t)x-x,
\]
  \item[d)] for every $x\in D(A)$ there is $Q(t)x\in D(A)$ and
\[
 \frac{d}{dt}Q(t)x=AQ(t)x=Q(t)Ax,
\]
\item[e)] for every $x\in D(A)$
\[
 Q(t)x-Q(s)x=\int_{s}^{t}Q(\tau)Axd\tau=\int_{s}^{t}AQ(\tau)xd\tau.
\]
\end{itemize}
\end{prop}

The operator $A^*:D(A^*)\rightarrow X$ is the adjoint of $A$. Important properties of the adjoint are summarized in the following remark \cite[Chapter 2.8]{Tucsnak_Weiss}.
\begin{rem}\label{rem properties of adjoint}
Let $A:D(A)\rightarrow X$ be a densely defined operator with $s\in\rho(A)$. The following holds
\begin{itemize}
 \item[(i)] If $A$ is closed (as $\rho(A)$ is not empty) we conclude that  $A^*$ is also closed, densely defined on $X$ and $A^{**}=A$.
 \item[(ii)] There is $\bar{s}\in\rho(A^*)$ and $[(sI-A)^{-1}]^*=(\bar{s}I-A^*)^{-1}$.
\item[(iii)] Let $(Q(t))_{t\geq0}$ be a strongly continuous semigroup on $X$. Then $(Q^*(t))_{t\geq0}$ is also a strongly continuous semigroup on $X$ and its generator is $A^*$.
\end{itemize}
 
\end{rem}

To overcome certain difficulties with unboundedness of the generator $A$, we make use of the duality with respect to a pivot space. In general, the idea of (duality with respect to) a pivot space can be described as follows. Having an unbounded closed linear operator $A:D(A)\rightarrow X$ with $D(A)\subset X$ densely, we want to establish a setting where it behaves much like a bounded one but defined on a Banach space. One instance of such situation is when we restrict ourselves to the space made out of its domain, but equipped with a graph (or graph-equivalent) norm. It is then reasonable to ask what is the dual of such space. It turns out that it can be represented as a completion of the original space $X$ with a resolvent--induced norm.  As the space $X$ is pivotal in the described setting, the name follows. A precise description of such situation can be found in \cite[Chapter 2.9]{Tucsnak_Weiss} or in \cite[Chapter II.5]{Engel_Nagel}

The following three propositions from \cite[Chapter 2.10]{Tucsnak_Weiss} introduce duality with respect to a pivot space  (sometimes referred to also as a rigged Hilbert space construction) in the context which we will use later.

\begin{prop}\label{prop X_1 space}
Let $A:D(A)\rightarrow X$ be a densely defined operator with $\rho(A)\neq\emptyset$. Then for every $\beta\in\rho(A)$ the space $(D(A),\|\cdot\|_{1})$, where
  \begin{equation}\label{eqn defn X_1 norm}
  \|z\|_{1}:=\|(\beta I-A)z\|_{X}\qquad \forall z\in D(A),
  \end{equation}
is a Hilbert space denoted $X_1$. The norms generated as above for different $\beta\in\rho(A)$ are equivalent to the graph norm. The embedding $X_{1}\subset X$ is continuous. If $Q(t)$ is the semigroup generated by $A$ then $Q(t)\in\mathcal{L}(X_{1})$ for every $t\in[0,\infty)$.

\end{prop}

For $A$ as in Proposition~\ref{prop X_1 space} its adjoint $A^*$ has the same properties. Thus, we can define the space $X_{1}^{d}:=(D(A^*),\|\cdot\|_{1}^{d})$ with the norm
  \begin{equation}\label{eqn defn X_1^d norm}
  \|z\|_{1}^{d}:=\|(\bar{\beta}I-A^*)z\|_{X}\qquad \forall z\in D(A^*),
  \end{equation}
where $\beta\in\rho(A)$, and this is also a Hilbert space.

\begin{prop}\label{prop completion of X}
 Let $A:D(A)\rightarrow X$ be a densely defined operator and let $\beta\in\rho(A)$. We denote by $X_{-1}$ the completion of $X$ with respect to the norm 
  \begin{equation}\label{eqn defn X-1 norm}
  \|z\|_{-1}:=\|(\beta I-A)^{-1}z\|_{X}\qquad \forall z\in X.
  \end{equation}
Then the norms generated as above for different $\beta\in\rho(A)$ are equivalent (in particular $X_{-1}$ is independent of the choice of $\beta$). Moreover, $X_{-1}$ is the dual of $X_{1}^{d}$ with respect to the pivot space $(X,\|\cdot\|_X)$.

The semigroup $(Q(t))_{t\geq0}$ generated by $A$ has a unique extension $({Q}_{-1}(t))_{t\geq0}$ such that ${Q}_{-1}(t)\in\mathcal{L}(X_{-1})$ for every $t\in[0,\infty)$.
\end{prop}

\begin{prop}\label{prop consequences of rigged Hilbert space}
Let $A:D(A)\rightarrow X$ be a densely defined operator with $\rho(A)\neq\emptyset$,  $\beta\in\rho(A)$,  $X_1$ be as in Proposition~\ref{prop X_1 space} and let $X_{-1}$ be as in Proposition~\ref{prop completion of X}. Then $A\in\mathcal{L}(X_1,X)$ and it has a unique extension $A_{-1}\in\mathcal{L}(X,X_{-1})$. Moreover,
\[
(\beta I-A)^{-1}\in\mathcal{L}(X,X_1),\qquad (\beta I-A_{-1})^{-1}\in\mathcal{L}(X_{-1},X)
\]
(in particular, $\beta\in\rho(A_{-1})$), and these two operators are unitary.
\end{prop}

\begin{rem}
In the remaining part we denote the extension $A_{-1}$ and the generator $A$ by the same symbol $A$. The same applies to the semigroup $(Q(t))_{t\geq0}$.
\end{rem}

Consider now a (linear) dynamical system described by the following initial value problem
\begin{equation}\label{eqn state equations}
    \begin{split}
            \frac{d}{dt}z(t)&=Az(t)+Bu(t)\\ 
            z(0)&\in X\qquad t\in J.  
    \end{split}
\end{equation}
where $X$ (called \textit{state space}) and $U$ (called \textit{control space}) are Hilbert spaces; $Z:=L_{loc}^1([0,\infty),X)\cap C([0,\infty),X_{-1})$ (called \textit{state trajectory space}) with $z\in Z$ and $u\in V:=L_{loc}^2([0,\infty),U)\cap C^1([0,\infty],U)$ (called \textit{control trajectory space}); $B\in \mathcal{L}(U,X_{-1})$; $Q(t)\in \mathcal{L}(X_{-1})$ for every $t\in J$ is an extension of a semigroup generated by $(A,D(A))$, $z_{0}:=z(0)\in X$.

The following Definition \cite[Definition 4.1.5]{Tucsnak_Weiss} is suitable in the context above, namely
\begin{defn}[Mild solution] \label{defn mild solution}
The $X_{-1}$-valued function $z$ defined by
\begin{equation}\label{eqn mild solution}
z(t):=Q(t)z_{0}+\int_{0}^{t}Q(t-s)Bu(s)ds,\quad t\in J
\end{equation}
is called the \textit{mild solution} of the corresponding differential equation~\eqref{eqn state equations}.
\end{defn}

The two basic types of controllability are given by the following
\begin{defn}[approximate controllability]\label{defn approx controllability}
The control process described by \eqref{eqn mild solution} is said to be \textit{approximately controllable} 
when for any given $z_{0},x_{T}\in X$ and any $\varepsilon>0$ there exists a control $u$ such that $\|z(T)-x_{T}\|\leq\varepsilon$, where $z_{0}$ is the initial condition and $u$ is the control. 
\end{defn}

\begin{defn}[exact controllability]\label{defn exact controllability}
The control process described by \eqref{eqn mild solution} is said to be \textit{exactly controllable} 
when $\varepsilon=0$ in definition~\ref{defn approx controllability}.
\end{defn}

In classical literature (see e.g. \cite{Tri75}), when no rigged Hilbert space construction was used, the following problem was of great importance. Namely, when taking equations \eqref{eqn state equations} as a primary model, its solution must lay in $D(A)$, which is only a dense subset of $X$. That means that the system \eqref{eqn state equations} cannot be exactly controllable. For the same reason, if considering infinite $T$ every approximately controllable system is exactly controllable. 

By the use of the rigged Hilbert space construction (called also ``a duality with respect to a pivot space``) the controllability problem is greatly simplified. Firstly, accodring to \cite[Proposition 4.1.4]{Tucsnak_Weiss} every solution to \eqref{eqn state equations} in $X_{-1}$ is a mild solution of \eqref{eqn state equations}. Although the converse, in general, still does not have to be true, due to the fact that now $A\in\mathcal{L}(X,X_{-1})$ greatly simplifies many considerations. 

This, however, comes at a price of the operator $B$ mostly being unbounded from $U$ to $X$. As we would like all the mild solutions \eqref{eqn mild solution} to be continuous $X$-valued functions, additional constraints must be put on the operator $B$. This is expressed by the following \cite[Definition 4.2.1]{Tucsnak_Weiss}
\begin{defn}\label{defn admissibility of B operator}
Let $B\in\mathcal{L}(U,X_{-1})$ and $\tau\geq0$. Define the operator $\Phi(\tau)$ as $\Phi(\tau)\in\mathcal{L}(L^2([0,\infty),U),X_{-1})$,
\begin{equation}\label{eqn defn forcing function}
 \Phi(\tau)u:=\int_{0}^{\tau}Q(\tau-\sigma)Bu(\sigma)d\sigma.
\end{equation}

The operator $B\in\mathcal{L}(U,X_{-1})$ is called an \textit{admissible control operator} for $(Q(t))_{t\geq0}$ if for some $\tau>0$ there is $\mathrm{Im}\Phi(\tau)\subset X$.
\end{defn}

\begin{rem}
Note that if $B$ is admissible, then in \eqref{eqn defn forcing function} we integrate in $X_{-1}$ but the integral is in $X$. Also, if the operator $\Phi(\tau)$ is such that $\mathrm{Im}\Phi(\tau)\subset X$ for some $\tau>0$ then for every $t\geq0$ there is $\Phi(t)\in\mathcal{L}(L^2([0,\infty),U),X)$ \cite[Proposition 4.2.2]{Tucsnak_Weiss}. Obviously every $B\in\mathcal{L}(U,X)$ is an admissible operator.
\end{rem}

The following Proposition \cite[Proposition 4.2.5]{Tucsnak_Weiss} shows that if $B$ is admissible and $u\in L_{loc}^2([0,\infty),U)$ then the initial value problem \eqref{eqn state equations} has a well-behaved unique solution in $X_{-1}$.
\begin{prop}\label{prop solution under admissibility of B}
 Assume that $B\in\mathcal{L}(U,X_{-1})$ is an admissible control operator for $(Q(t))_{t\geq0}$. Then for every $z_0\in X$ and every $u\in L_{loc}^2([0,\infty),U)$ the intial value problem \eqref{eqn state equations} has a unique solution in $X_{-1}$ given by \eqref{eqn mild solution} and it satisfies 
\[
 z\in C([0,\infty),X)\cap H_{loc}^1((0,\infty),X_{-1}).
\]
\end{prop}

\section{Controllability by Schmidt Theorem}
In this section we present our main findings. 

\subsection{Problem statement}
Consider the nonlinear dynamical system with zero initial condition stated by the differential equation in $X_{-1}$ as
$$\frac{d}{dt}z(t)=Az(t)+Bu(t)+f(z(t)),$$
where $u\in V=L_{loc}^2([0,\infty),U)\cap C^{\infty}([0,\infty),U)$, $A\in\mathcal{L}(X,X_{-1})$ and $B\in\mathcal{L}(U,X_{-1})$ is an admissible control operator for $(Q(t))_{t\geq0}$, $f:X\rightarrow X$ is a given continuous function. The mild solution of the above initial value problem is

\begin{equation}\label{eqn mild solution nonlinear}
    z(t)=\int_{0}^{t}Q(t-s)f\big(z(s)\big)ds+\int_{0}^{t}Q(t-s)Bu(s)ds,\quad t\in J.
\end{equation}
The main problem we tackle in this article is to find the conditions under which the dynamical system expressed by~\eqref{eqn mild solution nonlinear} is exactly controllable.

\subsection{Step 1}
To show the existence of a solution of problem \eqref{eqn mild solution nonlinear} we build an appropriate integral operator  $\Psi:Z\rightarrow Z$ and show that it has a fixed point. Let then
\begin{equation}\label{eqn integral Psi operator}
\Psi(z)(t):=\int_{0}^{t}Q(t-s)f\big(z(s)\big)ds+\int_{0}^{t}Q(t-s)Bu(s)ds,\quad t\in J.
\end{equation}

In Theorem~\ref{thm Schmidt ivp_existence} for $z$ to be a unique solution of the Cauchy problem stated there, $z$ has to be also a solution of the integral equation~\eqref{eqn Schmidt integral form}. What follows, the integral operator associated with \eqref{eqn Schmidt_IVP} has the form
\begin{equation}\label{eqn Schmidt by Psi operator}
\Psi(z)(t)=\int_{0}^{t}g\big(s,z(s)\big)ds+\int_{0}^{t}k\big(s,z(s)\big)ds,\quad t\in J.
\end{equation}

To show the existence of a fixed point of the operator~\eqref{eqn integral Psi operator} it is enough to show that appropriate parts of \eqref{eqn Schmidt by Psi operator} fulfil assumptions of Theorem~\ref{thm Schmidt ivp_existence}. Unfortunately, the obvious choice of functions under integrals in \eqref{eqn Schmidt by Psi operator}, namely $g(s,z(s)):=Q(t-s)f(z(s))$ and $k(s,z(s)):=Q(t-s)Bu(s)$ is not possible. The reason for that is that the semigroup $(Q(t))_{t\geq0}$ is defined for all $t\in[0,T]$, as well as functions $g,k:[0,T]\times X\rightarrow X$ in Theorem ~\ref{thm Schmidt ivp_existence}. Hence, for given nonlinearity $f$ we introduce formally two parameter-dependent families of functions 
    \begin{align*}
            &\mathcal{G}:=\{g_t:[0,t]\times X\rightarrow X:g_t(s,x):=Q(t-s)f(x),\ t\in [0,T]\}, \\
            &\mathcal{K}:=\{k_t:[0,t]\times X\rightarrow X:k_t(s,x):=Q(t-s)Bu_x(s),\ t\in [0,T]\},
    \end{align*}
where the steering trajectory $u_x$ is built based on an element $x$ of the state space, as explained below in \eqref{eqn steering trajectory}.Taking into account that members $g_t$ and $k_t$ of both families "work under the integral", the upper limit of which changes in the interval $[0,T]$, Theorem \ref{thm Schmidt ivp_existence} cannot be used directly. Instead, we will work it out from other facts. 

We make use of the following
\begin{defn}\label{defn Pickard operator et al}
Using the notation from Definition~\ref{defn mild solution} we define
    \begin{itemize}
    \item [a)] the Pickard-type \cite{Appell_2005} operator $L\in\mathcal{L}(Z)$, $$Lz:=\int_{0}^{\cdot}Q(\cdot-s)z(s)ds,$$
    \item [b)] the Pickard composition operator $L(t)\in\mathcal{L}(Z,X_{-1})$, $$L(t)z:=(Lz)(t),\quad t\in J,$$
    \item [c)] the nonlinear continuous composition $fz\in Z$, $(fz)(t):=f(z(t))$
    \end{itemize}
\end{defn}

With the above definition, the mild solution~\eqref{eqn mild solution nonlinear} may be rewritten in the form
\begin{equation}\label{eqn state in operator form}
z(t)=L(t)fz+L(t)Bu.
\end{equation}

Let $x_{T}\in X$ be the desired final state. Following a canonical procedure \cite{Tri75,CarQui84}, we assume exact controllability of the linear system without the nonlinear part $f$, given by Definition~\ref{defn mild solution}. Then, without loss of generality, we assume that the attainable set $\mathscr{A}_T$ is equal to the image of the $L(T)B$ operator, that is
\begin{align*}
\mathscr{A}_T:=&\{x_T\in X_{-1}: x_T=z(T),u\in V\}\\
=&\mathrm{Im}\big(L(T)B\big)=\mathrm{Im}\Phi(T)=X.
\end{align*}
The reason of such approach is to have a possibility to drive the system with nonlinear disturbance $f$ to every point it could attain without such disturbance. 

Define a linear and invertible operator $W:V/\ker\big(L(T)B\big)\rightarrow X$,
\begin{equation}\label{defn W operator}
W(u):=L(T)Bu.
\end{equation}
which has a bounded inverse operator $W^{-1}:X\rightarrow V/\ker\big(L(T)B\big)$, with $\|W^{-1}\|\leq M_{W^{-1}}<\infty$.


As we are interested in exact controllability, let us fix  $x_T\in X$. We construct a control signal based on this $x_T$ by selecting one element
\begin{equation}\label{eqn steering trajectory}
u_{x}\in W^{-1}\big(x_{T}-L(T)fz\big),
\end{equation}
which is explicitly related to a trajectory $z$ (with values in $X$ due to admissibility of $B$) which system \eqref{eqn mild solution nonlinear} will follow. 
Substituting control function defined by \eqref{eqn steering trajectory} into operator equation \eqref{eqn integral Psi operator} for $t=T$ we obtain
$$\Psi(z)(T)=L(T)fz+\big(L(T)B\big)\big(L(T)B\big)^{-1}\big(x_{T}-L(T)fz\big)=x_T.$$
By putting the same control function $u_{x}$ into mild solution \eqref{eqn mild solution nonlinear} we get $z(T)=x_T$, what gives $\Psi(z)(T)=z(T)$ and shows that the trajectory end point matches. The only thing left is to show that with the control function $u_x$ defined by \eqref{eqn steering trajectory} the operator $\Psi$ defined by \eqref{eqn integral Psi operator} has a fixed point in 
\[
Z= C([0,\infty),X)\cap\mathcal{H}_{loc}^1((0,\infty),X_{-1}),                                                                                                                                                                                                                                                                                                                                                                                                                                                                                                                                                                                                                                                                                                                                                                                                               \]
(note again that the operator $B$ is assumed to be admissible - Proposition \ref{prop solution under admissibility of B}) what means that there exists such trajectory $z$ of the system \eqref{eqn mild solution nonlinear} which leads it to the given final point $x_T$.

\subsection{Step 2}
The existence of a solution to integral equation \eqref{eqn Schmidt by Psi operator} is equivalent to the existence of a fixed point of the operator \eqref{eqn integral Psi operator}. We begin with the following

\begin{prop}\label{prop H1-H2}
Let $X$ be a real Hilbert space and we assume that
\begin{itemize}
    \item [(H1)]  $f:X\rightarrow X$ is continuous  on $X$ and there exists $M_f\in[0,\infty)$ such that $f$ fulfils a one-sided Lipschitz condition i.e.
\[
 \langle x_1 - x_2,f(x_1)-f(x_2)\rangle\leq M_f\norm{x_1-x_2}^2\qquad \forall x_1,x_2\in X,\ \forall t\in J,
\]
    \item [(H2)] there exists $M_g\in[0,\infty)$ such that for every $t\in J$ the map $g_t\in\mathcal{G}$ (i.e. $g_t:[0,t]\times X\rightarrow X$, $g_t(s,x):=Q(t-s)f(x)$) fulfills a one-sided Lipschitz condition
\[
 \langle x_1 - x_2,g_t(s,x_1)-g_t(s,x_2)\rangle\leq M_g\norm{x_1-x_2}^2\qquad \forall x_1,x_2\in X,\ \forall s\in [0,t].
\]
\end{itemize}
Then for a function $g:J\times X\rightarrow X$ given by $g(t,x):=\frac{d}{dt}\int_{0}^{t}g_t(s,x)ds$ we have
\[
[x_1-x_2,g(t,x_1)-g(t,x_2)]_{-}\leq M_g\norm{x_1-x_2}\qquad \forall x_1,x_2\in X,\ \forall t\in J.
\]
\end{prop}
\begin{proof}
\begin{itemize}
\item [1.] Fix $x\in X$ and define $G_{x}:J\rightarrow X$, $G_{x}(t):=\int_{0}^{t}g_t(s,x)ds=\int_{0}^{t}Q(t-s)f(x)ds$.
\item [2.] With the definition in 1. it follows that 
        \begin{equation*}
        \begin{split}
        &g(t,x)=\frac{d}{dt}G_{x}(t)=\lim_{h\rightarrow0}\frac{1}{h}\big(G_{x}(t+h)-G_{x}(t)\big)\\
        &=\lim_{h\rightarrow0}\frac{1}{h}\Big( \int_{0}^{t+h}Q(t+h-s)f(x)ds     -\int_{0}^{t}Q(t-s)f(x)ds \Big)\\
        &=\lim_{h\rightarrow0}\frac{1}{h}\Big( Q(h)\int_{0}^{t+h}Q(t-s)f(x)ds   -\int_{0}^{t}Q(t-s)f(x)ds \Big)\\
        &=\lim_{h\rightarrow0}\frac{1}{h}\Big((Q(h)-I)\int_{0}^{t}Q(t-s)f(x)ds  \\
        &\qquad\qquad\qquad +Q(h)\int_{t}^{t+h}Q(t-s)f(x)ds \Big)\\
        &=\lim_{h\rightarrow0}\Big(\frac{1}{h}(Q(h)-I)\int_{0}^{t}Q(t-s)f(x)ds  \\
        &\qquad\qquad\qquad +\frac{1}{h}\int_{t}^{t+h}Q(t+h-s)f(x)ds\Big)\\
        &=AG_{x}(t)+f(x)
        \end{split}
        \end{equation*}
        where the last equality is true provided that both limits on the right hand side exist. We show it below.
\item [3.] Fix $t\in J$. Then 
\begin{equation*}
 \begin{split}
    &\Big\|\frac{1}{h}\int_{t}^{t+h}Q(t+h-s)f(x)ds-f(x) \Big\|\\
    &=\Big\|\frac{1}{h}\int_{t}^{t+h}\big(Q(t+h-s)f(x)-f(x)\big)ds \Big\|\\
    &\leq\norm{Q(t+h-s)f(x)-f(x)}
 \end{split}
\end{equation*}
for some $s\in[t,t+h]$. As $h\rightarrow0$ there is also $s\rightarrow t$ and by strong continuity of the semigroup $Q(t)$ we have
\begin{equation*}\label{eqn result of prop H1-H2 point 2}
    \lim_{h\rightarrow0}\frac{1}{h}\int_{t}^{t+h}Q(t+h-s)f(x)ds=f(x),\quad \forall t\in J.
\end{equation*}
    
\item [4.] For any fixed $x\in X$, hence fixed $f(x)\in X$, by Proposition \ref{prop semigroup generator properties} there is 
\[
 \int_{0}^{t}Q(t-s)f(x)ds=\int_{0}^{t}Q(\tau)f(x)d\tau\in D(A),
\]
and we have 
\[
 A\int_{0}^{t}Q(t-s)f(x)ds=A\int_{0}^{t}Q(\tau)f(x)d\tau=Q(t)f(x)-f(x).
\]
In particular, although the integration is formally carried out in $X_{-1}$, the result is in $X$.
\item [5.] From points 3 and 4 it follows that 
\begin{equation}\label{eqn g computed}
g(t,x)=Q(t)f(x)
\end{equation}
is continuous and bounded.
\item [6.] Fix $t\in J$ and $x_1,x_2\in X$. We may write the following estimation
\begin{align*}
&\langle x_1-x_2,g(t,x_1)-g(t,x_2)\rangle=\langle x_1-x_2,Q(t)f(x_1)-Q(t)f(x_2)\rangle\\
&=\langle x_1-x_2,g_{t}(0,x_1)-g_{t}(0,x_2)\rangle\leq M_g\norm{x_1-x_2}^2.
\end{align*}
\item[7.] As $X$ is a Hilbert space over $\mathbb{R}$, the result of point 6 is equivalent, due to Lemma \ref{lem one-sided Lipschitz in Hilbert space}, to condtion a) of Theorem \ref{thm Schmidt ivp_existence}.
\end{itemize}
\end{proof}

Before proceeding further we state a useful lemma.
\begin{lem}\label{lem integral operator properties}
Let $X$ be a Banach space, $a\in X$, $\tau<T$ in $\mathbb{R}$, $J:=[\tau,T]$, $f:J\times X\rightarrow X$ continuous. For $u\in C(J,X)$ define
$\Phi: C(J,X)\rightarrow C(J,X)$,
\[
 (\Phi u)(t):=a+\int_{\tau}^{t}f(s,u(s))ds,\quad \forall t\in J.
\]
If the range  $\mathrm{Im}f\subseteq S\subseteq X $ then $(\Phi u)(t)\in a+(T-\tau)\cl\conv(S\cup\{0\})$.
\end{lem}
\begin{proof}
 \begin{itemize}
  \item [1.] Fix $u\in C(J,X)$. We have
    \[
    (\Phi u)(t)\approx a+\sum_{k=1}^{n}(t_k-t_{k-1})f(\tau_k,u(\tau_k)),
    \]
    where $\tau=t_0<t_1<t_2<\dots<t_n=t$.
  
  \item[2.] Rewriting above we get
    \[
    (\Phi u)(t)\approx a+(T-\tau)\Big[\sum_{k=1}^{n}\frac{t_k-t_{k-1}}{T-\tau}f(\tau_k,u(\tau_k))+\frac{T-t}{T-\tau}\Theta\Big],
    \]
    where $\Theta\in S\cup\{0\}$.
  \item[3.] As $\sum_{k=1}^{n}\frac{t_k-t_{k-1}}{T-\tau}+\frac{T-t}{T-\tau}=1$, there is 
    \[
  \Big[\sum_{k=1}^{n}\frac{t_k-t_{k-1}}{T-\tau}f(\tau_k,u(\tau_k))+\frac{T-t}{T-\tau}\Theta\Big]\in\conv(S\cup\{0\})\subseteq\cl\conv(S\cup\{0\}).
   \]
  \item[4.] As integral is a limit to the Riemann sums, each of which belongs to $\conv(S\cup\{0\})$, the integral itself belong to $\cl\conv(S\cup\{0\})$, i.e.
   \[
  (\Phi u)(t)\in a+(T-\tau)\cl\conv(S\cup\{0\}).
   \]
 \end{itemize}

\end{proof}

Let us now focus on assumption b) of Theorem~\ref{thm Schmidt ivp_existence}. We can relate it to our controllability setting by the following
\begin{prop}\label{prop H3-H5}
Using previously defined notation, if
 \begin{itemize}
    \item [(H3)] the operator $B\in\mathcal{L}(U,X)$ (hence, as bounded from $U$ to $X$, it is an admissible control operator for $(Q(t))_{t\geq0}$) and the linear system \eqref{eqn mild solution} is exactly controllable to the space $X=\big(L(T)B\big)$,
    \item [(H4)] the operator $W:V/\ker(L(T)B)\rightarrow X$, $W(u):=L(T)Bu$ has a bounded inverse operator $W^{-1}$,
    \item [(H5)] the space $X$ is ordered by a normal wedge $C$ and the operator $W^{-1}$ is such that for every $y\in X$ the function $f:[0,t]\rightarrow X$,
	  \[
	  f(s):=Q(t-s)BW^{-1}(y)(s)
	  \]
	  is convex ,
    \item [(H6)] there exists $M_w\in[0,\infty)$ such that
	  \[
	    \alpha(W^{-1}(D)(t))\leq M_w\alpha(D)\qquad \forall t\in J,\ \forall D\subset X,D\text{ bounded}, 
	  \]
    \item [(H7)] there exists $M_w'\in[0,\infty)$ such that
	  \[
	    \alpha(W^{-1}(D')(t))\leq M_w'\alpha(D)\qquad \forall t\in J,\ \forall D\subset X,D\text{ bounded}, 
	  \]
    where $D':=\{y\in X: W^{-1}(y)=\frac{d}{ds}W^{-1}(x),\ x\in D\}$,
 \end{itemize}
 then for a function $k:J\times X\rightarrow X$ given by  $k(t,x):=\frac{d}{dt}\int_{0}^{t}Q(t-s)BW^{-1}(x)(s)ds$ we have
 \[
 \alpha(k(J,D))\leq M_{k}\alpha(D)\qquad\forall D\subseteq X,\ D\ \text{bounded}.
 \]
\end{prop}

\begin{proof}
\begin{itemize}
\item [1.] From Definition \ref{defn Kuratowski mnc} for every $s\in[0,T]$ and every bounded $D\subset X$ we have
  \begin{align*}
    \alpha\big(W^{-1}(D)(s)\big)=&\inf\big\{\delta(s)\geq0:W^{-1}(D)(s)\subset\bigcup_{i=1}^{n}\Sigma_{i}(s);\\
    &\diam\Sigma_{i}(s)\leq\delta(s),\ \Sigma_{i}(s)\subset U,\ i\in\{1,\dots,n\};n\in\mathbb{N}\big\}
  \end{align*}
where $\Sigma_{i}\subset V/\ker(L(T)B)=L_{loc}^{2}([0,\infty),U)\cap C^\infty([0,\infty),U)/\ker(L(T)B)$.
\item[2.] Further, for every $\tau\in[0,T]$ we have
  \begin{align*}
    \alpha\big(Q(\tau)BW^{-1}(D)(s)\big)=&\inf\big\{\delta'(s)\geq0:Q(\tau)BW^{-1}(D)(s)\subset\bigcup_{i=1}^{n}Q(\tau)B\Sigma_{i}(s);\\
    &\diam Q(\tau)B\Sigma_{i}(s)\leq\delta'(s)\big\}
  \end{align*}
  and
  \[
    \diam(\Sigma_{i}(s))=\sup_{u(s),v(s)\in\Sigma_{i}(s)}\norm{u(s)-v(s)}.
  \]
  Let us fix $u(s),v(s)\in\Sigma_{i}(s)$ and let $x(s):=Q(\tau)Bu(s)$, $y(s):=Q(\tau)Bv(s)$. We then have $x(s),y(s)\in Q(\tau)B\Sigma_{i}(s)$ and
  \[
    \norm{x(s)-y(s)}=\norm{Q(\tau)B(u(s)-v(s))}\leq\norm{Q(\tau)B}\norm{u(s)-v(s)}.
  \]
  Hence $\diam(Q(\tau)B\Sigma_{i}(s))\leq M_q\norm{B}\diam(\Sigma_{i}(s))$, where $M_q:=\max_{t\in J}\norm{Q(t)}$. Using point 1 we may now write
  \[
    \diam(Q(\tau)B\Sigma_{i}(s))\leq \delta'(s)\leq M_q\norm{B}\diam(\Sigma_{i}(s))\leq M_q\norm{B}\delta(s)
  \]
  for suitably chosen $\delta'(s),\delta(s)$. Passing to infimum we get the estimation
  \[
    \alpha\big(Q(\tau)BW^{-1}(D)(s)\big)\leq M_q\norm{B}\alpha(W^{-1}(D)(s))
  \]
  for every $s,\tau\in J$ and every bounded $D\subset X$.
\item[3.] Using now assumption $(H6)$ we obtain
  \[
   \alpha\big(Q(\tau)BW^{-1}(D)(s)\big)\leq M_q\norm{B}M_w\alpha(D)
  \]
  for every $s,\tau\in J$ and every bounded $D\subset X$.

\item[4.] Preparing the ground for the Ambrosetti Theorem \ref{thm Ambrosetti's MNC} let $J_t:=[0,t]\subset[0,T]$ and define a family of operators indexed by the elements of $D\subset X$, namely
  \begin{equation}\label{eqn defn equicontinuous familty F_t}
   \mathcal{F}_{t}:=\{Q(t-\cdot)BW^{-1}(y)(\cdot)\}_{y\in D}\subset C(J_t,X).
  \end{equation}
We will show that for every $t\in J$ and every bounded $D\subset X$ the family $\mathcal{F}_{t}\subset C(J_t,X)$ is equicontinuous. For that purpose note firstly that the operator $W^{-1}$ is a bounded and linear operator, hence it is continuous on $X$. 

   
Now fix $t\in J$ and define a function $\varphi_{t}:J_{t}\times X\rightarrow X$, 
\begin{equation}\label{eqn defn supplementary varphi function}
\varphi_{t}(s,y):=Q(t-s)BW^{-1}(y)(s).
\end{equation}
We will show that $\varphi_t$ is continuous on the product $J_t\times X$. Fix $y\in X$ and let $s_1,s_2\in J_t$ be such that $s_1+\tau=s_2$, $\tau>0$. We then have
\begin{align*}
  &\norm{Q(t-s_1)BW^{-1}(y)(s_1)-Q(t-s_2)BW^{-1}(y)(s_2)}\\
 &=\norm{Q(t-s_2+\tau))BW^{-1}(y)(s_2-\tau)-Q(t-s_2)BW^{-1}(y)(s_2)}\\
 &=\norm{Q(t-s_2)Q(\tau)BW^{-1}(y)(s_2-\tau)-Q(t-s_2)BW^{-1}(y)(s_2)}\\
 &=\norm{Q(t-s_2)\Big(Q(\tau)BW^{-1}(y)(s_2-\tau)-Q(\tau)BW^{-1}(y)(s_2)+\\
 &\quad+Q(\tau)BW^{-1}(y)(s_2)-BW^{-1}(y)(s_2)\Big)}\\
 &\leq\norm{Q(t-s_2)}\norm{Q(\tau)}\norm{BW^{-1}(y)(s_2-\tau)-BW^{-1}(y)(s_2)}+\\
 &\quad+\norm{Q(t-s_2)}\norm{Q(\tau)BW^{-1}(y)(s_2)-BW^{-1}(y)(s_2)},
\end{align*}
 where the last part tends to $0$ with $s_1\rightarrow s_2$, that is with $\tau\rightarrow0$. This follows from the continuity of $W^{-1}(y)$ on $J_t$ and strong continuity of the semigroup $Q$. Now joint continuity of $\varphi_t$ follows from linearity and continuity of $W^{-1}$ on $X$ and the decomposition 
 \[
 \varphi_t(s,y)-\varphi_t(\tau,z)=\varphi_t(s,y)-\varphi_t(s,z)+\varphi_t(s,z)-\varphi_t(\tau,z),
 \]
 where $(\tau,z)\rightarrow(s,y)$. 
 
 From continuity of $\varphi_t$ it follows that for every bounded $D\subset X$ the set $\varphi_t(J_t,D)\subset X$ remains bounded, i.e. there exists such $r<\infty$ that $\varphi_t\subset B(0,r)$, a zero-centred ball with a finite radius $r$. Defining, for a given bounded $D\subset X$, the set
 \begin{equation}\label{eqn defn equicontinuous familty values F_t(s)}
   \mathcal{F}_{t}(s):=\{\theta(s):\theta\in\mathcal{F}_t\}
  \end{equation}
  we have
 \[
  \mathcal{F}_{t}(s)=\bigcup_{y\in D}\varphi_t(s,y)\subset B(0,r)
 \]
 for suitable $r<\infty$. By $(H5)$ and Theorem~\ref{thm equicontinuity theorem by Kosmol} it follows that the collection of continuous mappings $\mathcal{F}_t$ is equicontinuous for every $t\in J$ and every bounded $D\subset X$.
 
 \item[5.] Fix bounded $D\subset X$. From the Ambrosetti Theorem \ref{thm Ambrosetti's MNC} and point 4 we have 
 \[
 \alpha(\mathcal{F}_{t})=\sup_{s\in J_t}\alpha(\mathcal{F}_t(s))=\alpha(\mathcal{F}_t(J_t))\qquad\forall t\in J,
 \]
 where $\mathcal{F}_t(J_t)=\bigcup_{s\in J_t}\mathcal{F}_t(s)$.
Point 3 gives
\[
\alpha(\mathcal{F}_t(s))=\alpha\big(Q(t-s)BW^{-1}(D)(s)\big)\leq M_q\norm{B}M_w\alpha(D)\quad\forall t\in J\ \forall s\in J_t.
\] 
 In consequence we have
\begin{equation}\label{eqn Ambrosetti thm for equicontinuous F_t}
 \sup_{s\in J_t}\alpha(\mathcal{F}_t(s))=\alpha(\mathcal{F}_t(J_t))\leq M_q\norm{B}M_w\alpha(D)\quad\forall t\in J.
\end{equation}


Note that for every $t\in J$ and every $y\in D$ there is $\varphi_t(s,y)\in\mathcal{F}_t(s)$. Hence, $\varphi_t(J_t,D)=\{\varphi(s,y):s\in J_t,y\in D\}=\mathcal{F}_t(J_t)$ and for every $t\in J$ and every bounded $D\subset X$ we get
\[
\alpha(\varphi_t(J_t,D))\leq M_k\alpha(D),
\]
hence for every $t\in J$ the mapping $\varphi_t:J_t\times X$ is condensing with constant $M_k:=M_q\norm{B}M_w$.

\item[6.] Define $K:J\times X\rightarrow X$, $K(t,x):=\int_{0}^{t}\varphi_t(s,x)ds=\int_{0}^{t}Q(t-s)BW^{-1}(x)(s)ds$. Fix $t\in J$ and $x\in X$, then
\begin{equation}\label{eqn k computed}
\begin{split}
 k(t,x)&=\frac{d}{dt}K(t,x)=\lim_{h\rightarrow0}\frac{1}{h}\big(K(t+h,x)-K(t,x)\big)\\
        &=\lim_{h\rightarrow0}\frac{1}{h}\Big( \int_{0}^{t}Q(t+h-s)BW^{-1}(x)(s)ds \\
	&\quad -\int_{0}^{t}Q(t-s)BW^{-1}(x)(s)ds\\
        &\quad+\int_{t}^{t+h}Q(t+h-s)BW^{-1}(x)(s)ds \Big)\\
        &=\lim_{h\rightarrow0}\Big(\frac{1}{h}(Q(h)-I)\int_{0}^{t}Q(t-s)BW^{-1}(x)(s))ds\\
        &\quad+\frac{1}{h}\int_{t}^{t+h}Q(t+h-s)BW^{-1}(x)(s)ds\Big)\\
        &=AK(t,x)+BW^{-1}(x)(t)+\int_{0}^{t}Q(t-s)B\frac{d}{ds}W^{-1}(x)(s)ds\\
	&=Q(t)BW^{-1}(x)(0)+L(t)Bu_{x}',
\end{split}
\end{equation}
provided that appropriate limits exist. We show it below.

\item [7.] Consider again the function $\varphi_{t}$ defined in \eqref{eqn defn supplementary varphi function}. Calulating its time derivative at $s\in[0,t]$ we obtain
\[
 \frac{d}{ds}\varphi_{t}(s,x)=-AQ(t-s)BW^{-1}(x)(s)+Q(t)B\frac{d}{ds}W^{-1}(x)(s).
\]
Initially the above result can be found either by elementary limit calculation or one can use the result in \cite[Lemma B.16]{Engel_Nagel}. Here, both parts on the right hand side exist, although for the sake of clarity we skip all the routine limit considerations in the argument of the generator $A$ leading to application of its extension - for more details see Proposition \ref{prop consequences of rigged Hilbert space} and \cite[Proposition 2.10.3]{Tucsnak_Weiss}. Note also, that by assumption the function 
\begin{equation}\label{eqn defn steering control time derivative}
u_{x}':[0,t]\rightarrow U,\qquad u_{x}':=\frac{d}{ds}W^{-1}(x)
\end{equation}
exists for all $x\in X$ and is continuous.

We also have the following 
\begin{align*}
 &BW^{-1}(x)(t)-Q(t)BW^{-1}(x)(0)=\varphi_{t}(t,x)-\varphi_{t}(0,x)=\int_{0}^{t}\frac{d}{ds}\varphi_{t}(s,x)ds\\
 =&-\int_{0}^{t}A\varphi_{t}(s,x)ds+L(t)Bu_{x}'=-A\int_{0}^{t}\varphi_{t}(s,x)ds+L(t)Bu_{x}',
\end{align*}

where, due to Proposion \ref{prop semigroup generator properties} we can move from the extension to the original generator $A$.
In consequence
\begin{align*}
 AK(t,x)&=A\int_{0}^{t}\varphi_{t}(s,x)ds=A\int_{0}^{t}Q(t-s)BW^{-1}(x)(s)ds\\
&=Q(t)BW^{-1}(x)(0)-BW^{-1}(x)(t)+L(t)Bu_{x}'.
\end{align*}

\item[8.] To finish the proof of \eqref{eqn k computed} consider the following estimaiton
\begin{align*}
        &\Big\|\frac{1}{h}\int_{t}^{t+h}Q(t+h-s)BW^{-1}(x)(s)ds-BW^{-1}(x)(t)\Big\|\\
        &=\frac{1}{h}\Big\|\int_{t}^{t+h}\Big(Q(t+h-s)BW^{-1}(x)(s)-BW^{-1}(x)(t)\Big)ds\Big\|\\
        &\leq\|Q(t+h-s)BW^{-1}(x)(s)-Q(t+h-s)BW^{-1}(x)(t)\|\\
	&+\|Q(t+h-s)BW^{-1}(x)(t)-BW^{-1}(x)(t)\|,
\end{align*}
for some $s\in[t,t+h]$. Taking the limit as $h\rightarrow 0$ there is also $s\rightarrow t$ and due the continuity of $t\mapsto Q(t)$ and continuity of $W^{-1}(x)$ above tends to zero and we obatin
\[
 \lim_{h\rightarrow0}\frac{1}{h}\int_{t}^{t+h}Q(t+h-s)BW^{-1}(x)(s)ds=BW^{-1}(x)(t).
\]
Combining now this result and the one of point 7 we obtain \eqref{eqn k computed}.

\item[9.]Fix bounded $D\subset X$. We have
\[
 k(t,x)=Q(t)BW^{-1}(x)(0)+L(t)Bu_{x}'=k_t(0,x)+L(t)Bu_{x}'
\]
and
\begin{align*}
k(J,D)=\bigcup_{x\in D}\bigcup_{t\in J}k(t,x).
\end{align*}
Note that for every $t\in J$ and every $x\in D$ there is 
$Q(t)BW^{-1}(x)(0)\in\mathcal{F}_t(0)$ with $\mathcal{F}_t$ defined for the same index set $D$. 

Due to the fact that $W$ is an injection, for every $x\in X$ there exists a unique $y\in X$ such that 
\[
 W^{-1}(y)=u_{x}'=\frac{d}{ds}W^{-1}(x)\in V/\ker L(T)B
\]
and $\norm{u_{x}'}=\norm{W^{-1}(y)}<\infty$.
In consequence, for every bounded $D\subset X$ the set 
\[
 D':=\{y\in X: W^{-1}(y)=u_{x}'=\frac{d}{ds}W^{-1}(x),\ x\in D\}
\]
is unique. 

Define, similarly to point 4, the equicontinuous family of operators 
  \begin{equation}\label{eqn defn equicontinuous familty F_t'}
    \begin{split}
    \mathcal{F}_{t}':&=\{Q(t-\cdot)BW^{-1}(y)(\cdot)\}_{y\in D'}\\
		    &=\{Q(t-\cdot)Bu_{x}'\}_{x\in D}\subset C(J_t,X).
    \end{split}
  \end{equation}
which may be regarded as indexed by elements of either $D'$ or $D$. Using the assumption $(H7)$ and following the same procedure which led to \eqref{eqn Ambrosetti thm for equicontinuous F_t}, we have
\begin{equation}\label{eqn Ambrosetti thm for equicontinuous F_t'}
 \alpha(\mathcal{F}_t'(J_t))\leq M_q\norm{B}M_w'\alpha(D)
\end{equation}

As the range of the function $[0,t]\ni s\mapsto Q(t-s)Bu_{x}'(s)\in X$ is contined in $\mathcal{F}_{t}'(J_t)$,
according to Lemma \ref{lem integral operator properties} there is 
\[
 L(t)Bu_{x}'\in T\cl\conv(\mathcal{F}_{t}'(J_t)\cup\{0\}).
\]

Define now a collection of operators $\mathcal{P}:=\{k(\cdot,x)\}_{x\in D}$, where each member acts from $J$ to $X$. From point 4 and above considerations we see that  $\mathcal{P}$ is in fact a collection of bounded operators, indexed again by elements of $D\subset X$. With the same reasoning as in point 4 we see that $\mathcal{P}$ is an equicontinuous set and, by Ambrosetti Theorem \ref{thm Ambrosetti's MNC}, we have
\begin{equation}\label{eqn Ambrosetti thm for equicontinuous P}
 \alpha(\mathcal{P})=\sup_{t\in J}\alpha(\mathcal{P}(t))=\alpha(\mathcal{P}(J))\qquad\forall t\in J.
\end{equation}
Due to the definition of $\mathcal{P}$ we have $k(J,D)=\mathcal{P}(J)$ and 
\[
 \mathcal{P}(t)=\bigcup_{x\in D}Q(t)BW^{-1}(x)(0)+L(t)Bu_{x}'\subset \mathcal{F}_t(J_t)\cup T\cl\conv(\mathcal{F}_{t}'(J_t)\cup\{0\})
\]
for every $t\in J$.
From \eqref{eqn Ambrosetti thm for equicontinuous F_t}, \eqref{eqn Ambrosetti thm for equicontinuous F_t'} and \eqref{eqn Ambrosetti thm for equicontinuous P} and Theorems \ref{thm mnc properties} and \ref{thm mnc properties B-space} it follows now that 
\begin{align*}
  \alpha(k(J,D))=&\alpha(P(J))\leq\alpha\big(\mathcal{F}_t(J_t)\cup T\cl\conv(\mathcal{F}_{t}'(J_t)\cup\{0\})\big)\\
		\leq& \max\{M_w,M_w'\}TM_q\norm{B}\alpha(D)
\end{align*}
and function $k$ is condensing.

\end{itemize}

\end{proof}

Based on Proposition \ref{prop H1-H2} and Proposition \ref{prop H3-H5} we may state the main Theorem of this article which gives sufficient conditions for the existence of a solution to integral equation \eqref{eqn Schmidt by Psi operator}. This is equivalent to the existence of a fixed point of the operator \eqref{eqn integral Psi operator} and results in exact controllability of system \eqref{eqn mild solution nonlinear}.

\begin{thm}\label{thm main theorem}
Assume $(H1)-(H7)$. Then a dynamical system with mild solution given by \eqref{eqn mild solution nonlinear} is exaclty controllable to the space $\mathrm{Im}\big(L(T)B\big)=X$, with trajectory $z\in C([0,\infty),X)\cap\mathcal{H}_{loc}^1((0,\infty),X_{-1})$.
\end{thm}
\begin{proof}
 The proof follows immediately from the reasoning in Step 1 section, Proposition \ref{prop H1-H2}, \ref{prop H3-H5}  and application of Theorem \ref{thm Schmidt ivp_existence}.
\end{proof}

\section{Example}
Let us take an example similar to the one chosen in \cite{RavMach13}, but with an emphasis put on nonlinearity $f$. Consider a one dimensional real non homogeneous transport partial differential equation with nonlinear part, given by
\begin{equation}\label{eqn transport PDE}
    \begin{split}
    &\frac{\partial}{\partial t}z(t,\xi)=\frac{\partial}{\partial\xi}z(t,\xi)+m(\xi)u(t,\xi)+f(z(t,\xi)),\\
    &z(0,\xi)=0\in X,
    \end{split}
\end{equation}
where spatial coordinate $\xi\in[0,1]$, time coordinate $t\in J:=[0,T]$, the state space $X$ and control space $U$ be $L^{2}(0,1)$. In other words, we consider the mapping $z$ defined on the cartesian product $[0,T]\times [0,1]$ as the state trajectory which for every $t\in J$ takes value $z(t)$, which is a mapping of $L^{2}$ class from the interval $[0,1]$ to $\mathbb{R}$.

Let $A:D(A)\rightarrow X$, $D(A)\subset X$ densely, be a generator of a vanishing (or nilpotent) left shift semigroup \cite[Example 2.3.8]{Tucsnak_Weiss}, defined as a spatial differentiation operator 
\[
Ax:=\frac{d}{d\xi}x,\qquad D(A):=\{x\in\mathcal{H}^{1}(0,1):x(1)=0\},
\]
where $\mathcal{H}^{1}(0,1)$ is the Sobolev space of all $L^{2}(0,1)$ functions for which its first derivative is also square integrable \cite[Definition 5.2.2]{Evans}. The semigroup $(Q(t))_{t\geq0}$ is explicitly given by

\begin{equation}\label{defn shift semigroup}
\big(Q(t)x\big)(\xi):=
\left\{\begin{array}{ll}
        x(\xi+t)    & \textrm{if } \xi\in[0,1],\ \xi+t\leq 1,\\
        0           & \textrm{if } \xi\in[0,1],\ \xi+t>1,\\
       \end{array}
\right.
\end{equation}
where we take $t\in J$. The semigroup $(Q(t))_{t\geq0}$ is not compact on $X$, but $\alpha(Q(t)D)\leq2\alpha(D)$ for every bounded set $D\subset X$, making it a condensing operator. Note also, that the semigroup $(Q(t))_{t\geq0}$ is contractive.

Define the control operator $B\in\mathcal{L}(U,X)$ appropriately as
\[
Bu(t)(\xi):=m(\xi)u(t)(\xi)
\]
for every $t\in J$ and $\xi\in[0,1]$ where $m(\xi)$ provides a spatial distribution of control. 

It is known that Hilbert spaces $L^2(0,1)$ and $l^2$ are isometrically isomorphic \cite{Muscat}. Let $P:L^2(0,1)\rightarrow l^2$ be such isometric isomorphism. Define also continuous $\varphi:\mathbb{R}\rightarrow\mathbb{R}$ as
\[
\varphi(\xi):=
\left\{\begin{array}{ll}
        0    						& \textrm{if } \xi<0,\\
        - \sqrt{\xi}         & \textrm{if } \xi\in[0,1],\\
        -1						& \textrm{if } \xi>1
       \end{array}
\right.
\]
and $\rho:l^2\rightarrow l^2$, $\rho(\alpha)=\rho\big((\alpha^1,\alpha^2,\alpha^3,\dots)\big):=(\varphi(\alpha^1),\frac{1}{2}\varphi(\alpha^2),\frac{1}{3}\varphi(\alpha^3),\dots)$. Let now the nonlinearity $f:X\rightarrow X$ be given by
\[
f(x):=(P^{-1}\rho P)(x).
\]
The dissipativity condition from assumption  $(H1)$ in Proposition \ref{prop H1-H2}, that is 
\[
 \langle x - y,f(x)-f(y)\rangle\leq M_f\norm{x-y}^2\qquad \forall x,y\in X
\]
is equivalent to 
\begin{equation}\label{eqn dissipativity cond equivalent}
\langle Px - Py,\rho Px-\rho Py\rangle\leq M_f\norm{Px-Py}^2\qquad \forall x,y\in X.
\end{equation}
We will show that $f$ does not fulfil Lipschitz condition
\begin{equation}
\exists_{M_f<\infty}\ \forall_{x,y\in L^{2}(0,1)}\quad\norm{f(x)-f(y)}\leq M_f\norm{x-y},
\end{equation}
which, using the definition of $f$, is equivalent to 
\begin{equation}\label{eqn defn Lipschitz condition equivalent}
\exists_{M_f<\infty}\ \forall_{x,y\in L^{2}(0,1)}\quad\norm{\rho Px-\rho Py}\leq M_f\norm{Px-Py}.
\end{equation}
 Indeed, fix $y=0$ and such sequence $(x_m)_{m\in\mathbb{N}}$ of elements of $L^2(0,1)$ that $(\alpha_m)_{m\in\mathbb{N}}:=(Px_m)_{m\in\mathbb{N}}$ and $\alpha_m=(\frac{1}{m},0,0,\dots)$, $m\in\mathbb{N}$. Note that $\alpha_m\rightarrow 0=Py$ as $m\rightarrow\infty$. We then have
 \[
\frac{\norm{\rho(Px_m)}}{\norm{Px_m}}=\frac{\norm{\rho(\alpha_m)}}{\norm{\alpha_m}}
=\frac{\sqrt{\varphi^2(\alpha^1)}}{\sqrt{\frac{1}{m^2}}}=\frac{\sqrt{\frac{1}{m}}}{\frac{1}{m}}=\sqrt{m}
 \]
and as $\sqrt{m}\rightarrow\infty$ as $m\rightarrow\infty$ we see that Lipschitz condition \eqref{eqn defn Lipschitz condition equivalent} cannot be fulfilled at $y=0$.

It remains to show that $f$ fulfils condition \eqref{eqn dissipativity cond equivalent}. Fix $Px=\alpha$ and $Py=\beta$. We have
\begin{align*}
	&\langle\alpha-\beta,\rho(\alpha)-\rho(\beta)\rangle=\sum_{i=1}^{\infty}(\alpha^i-\beta^i)(\varphi(\alpha^i)-\varphi(\beta^i))\leq0
\end{align*}
because, due to monotonicity of $\varphi$, the element $(\alpha^i-\beta^i)(\varphi(\alpha^i)-\varphi(\beta^i))\leq0$ for every $i\in\mathbb{N}$. Hence, it is enough to take any positive $M_f$ in \eqref{eqn dissipativity cond equivalent} and the condition is met. Note also that with the semigroup Definition \eqref{defn shift semigroup} assumption  $(H2)$ is equivalent to \eqref{eqn dissipativity cond equivalent}.

Note also that $f$ is uniformly bounded. This follows from the fact that 
\[
\norm{f(x)}=\norm{\rho(Px)}\qquad\forall x\in X
\]
and
\[
\norm{\rho(\alpha)}^2=\sum_{n=1}^{\infty}\big[\frac{1}{n}\varphi(\alpha^n)\big]^2\leq\sum_{n=1}^{\infty}\frac{1}{n^2}<\infty.
\]

\section{Conclusions}
In this article we showed new results in establishing sufficient conditions for controllability of particular types of dynamical systems. Our results expand the results found in \cite{CarQui84}, where the authors use Nussbaum fixed point theorem. In particular, we did not impose any compactness condition on the semigroup, instead we used its condensing property. This is a considerably weaker assumption than the one taken in the above mentioned work.

The second improvement in comparison to the present state of literature is that we did not assume that the nonlinearity is Lipschitz. The price paid for that is that we used an existence result initially intended for the initial value problem, not formulated in a fixed point form. The authors are not aware whether there exists a similar fixed point theorem.

Our result can be expanded to incorporate phenomena such as impulsive behaviour or nonlocal conditions in a way similar to \cite{RavMach13}. Note, however, that the assumptions of the Schmidt theorem are in a sense weaker than the demands of the M{\"o}nch's condition used be the authors of \cite{RavMach13}.

\bibliographystyle{amsplain}

\paragraph{Acknowledgements}
This project has received funding from the European Union's Horizon 2020 research and innovation programme under the Marie Sk{\l}odowska-Curie grant agreement No 700833.\\

\bibliography{ref_database}

\end{document}